\documentclass[apaper,11pt]{article}
\usepackage{mathrsfs}
\usepackage{color}
\usepackage{amsmath, epsfig, cite}
\usepackage{amsthm}
\usepackage{amssymb}
\usepackage{amsfonts}
\usepackage{latexsym}
\usepackage{graphicx}
\usepackage{multirow}
\usepackage{graphics,epsfig,psfrag}
\usepackage{epstopdf}
\usepackage{caption}
\usepackage{float}
\usepackage{pifont}
\usepackage{subcaption}
\usepackage{indentfirst}
\usepackage{tabularx}
\usepackage{booktabs}
\usepackage{array}
\usepackage[table]{xcolor} 
\usepackage{colortbl}      

\newtheorem{thm}{Theorem}

\newtheorem{Lem}{Lemma}[section]

\newtheorem{obs}{Observation}[section]

\definecolor{LightBlue}{RGB}{240, 245, 255}
\definecolor{LightOrange}{RGB}{255, 245, 240}
\definecolor{LightGreen}{RGB}{240, 255, 250}
\definecolor{LightYellow}{RGB}{255, 255, 240}

\makeatletter \@addtoreset{equation}{section} \makeatother

\title{\bf Reliability evaluation of Cayley graph generated by unicyclic graphs based on cyclic fault pattern}
\author{
Ting Tian$^{1}$,\ \ Shumin Zhang$^{1,2,3,}$\thanks{Corresponding author. \newline\indent{\emph{~E-mail}: {\tt tianting1201@163.com} (T.\ Tian), {\tt zhangshumin@qhnu.edu.cn} (S.\ Zhang), {\tt zhuboqh@163.com} (B.\ Zhu)}
}\ ,\ \ Bo Zhu$^{1}$ \\ \\ [0.2cm]
\small $^{1}$School of Mathematics and Statistics, Qinghai Normal University, Xining, Qinghai 810008, China\\[0.2cm]
\small $^{2}$Academy of Plateau Science and Sustainability, People's Government of Qinghai Province and\\
\small Beijing Normal University, Xining, Qinghai 810008, China\\[0.2cm]
\small $^3$The State Key Laboratory of Tibetan Intelligent Information Processing and Application,\\
\small Xining, Qinghai 810008, China\\[0.2cm]
\date{}
}

\begin{document}

\maketitle
\baselineskip=0.23in


\begin{abstract}

Graph connectivity serves as a fundamental metric for evaluating the reliability and fault tolerance of interconnection networks. To more precisely characterize network robustness, the concept of cyclic connectivity has been introduced, requiring that there are at least two components containing cycles after removing the vertex set. This property ensures the preservation of essential cyclic communication structures under faulty conditions.
Cayley graphs exhibit several ideal properties for interconnection networks, which permits identical routing protocols at all vertices, facilitates recursive constructions, and ensures operational robustness. In this paper, we investigate the cyclic connectivity of Cayley graphs generated by unicyclic triangle free graphs. Given an symmetric group $Sym(n)$ on $\left\{ 1,2,\dots,n\right\}$ and a set $\mathcal{T}$ of transpositions of $Sym(n)$. Let $G(\mathcal{T})$ be the graph on vertex set $\left\{ 1,2,\dots,n\right\}$ and edge set $\left\{ij\colon(ij)\in \mathcal{T}\right\}$. If $G(\mathcal{T})$ is a unicyclic triangle free graphs, then denoted the Cayley graph Cay$(Sym(n),\mathcal{T})$ by $UG_{n}$. As a result, we determine the exact value of cyclic connectivity of $UG_{n}$ as $\kappa_{c}(UG_{n})=4n-8$ for $n\ge 4 $.


\end{abstract}

\newpage

\parskip 4pt
\baselineskip=0.22in
\section{Introduction}

The topological structure of interconnection networks is conventionally modeled using graphs, where vertices represent processing nodes and edges denote communication links among vertices. To quantitatively evaluate reliability and fault tolerance of networks, various graph-theoretic parameters have been developed, among which connectivity measures play a fundamental role in assessing network robustness. Classical connectivity is the minimum cardinality of vertex sets whose removal leaves the remaining graph disconnected or trivial. However, this conventional approach assumes the unrealistic scenario where all neighbors of a vertex fail simultaneously, which is a worst case measure and thereby underestimating the practical resilience of most interconnection networks.

To address this limitation, Harary~\cite{1} introduced conditional connectivity, which requires that every surviving vertex maintains at least one faut-free neighbor after the vertex removal. This refined concept overcomes the constraints of classical connectivity and provides a more accurate characterization of fault tolerance in real-world settings. In many critical applications, it is essential that surviving network components not only remain connected but also preserve specific structural properties. This requirement has motivated the development of specialized connectivity measures under various constraints. Notable among these are $g$-good-neighbor connectivity, $h$-extra connectivity, and $r$-component connectivity, which have attracted substantial research interest~\cite{2,3,4,5,6,7,8,9,10,11,12}. The progression from classical connectivity to conditional connectivity parameters represents a significant advancement in network reliability analysis, enabling more nuanced and practical assessment of fault tolerance in complex interconnection systems. These refined metrics provide valuable insights for designing robust network architectures capable of maintaining operational integrity under partial failure conditions.

The removal of any single vertex from a cycle does not break the connection between the vertices in that cycle. By providing built-in path redundancy, cycles ensure continuous data routing in communication systems and computational clusters when individual vertices fail. This inherent structural resilience allows damaged networks to maintain operational continuity through dynamic path reconfiguration, which demonstrates their practical value in real-world networks. Robertson~\cite{13} was the first to introduce the concept of cyclic connectivity, the \emph{cyclic connectivity} of $G$, denoted by $\kappa_{c}(G)$, is defined as the smallest cardinality of a vertex subset $S$ of a graph $G$ such that the graph obtained by removing $S$ from $G$ (denoted $G-S$) is disconnected and has at least two components with cycles. Recent findings related to cyclic connectivity can be found in references~\cite{14,15,16,17,18,19}. The study of such cycle-rich residual graphs provides actionable strategies for designing cost-effective fault-tolerant architectures, particularly for applications requiring uninterrupted service in industrial control and distributed computing environments.

Cayley graphs exhibit several ideal properties for interconnection networks, which permits identical routing protocols at all vertices, facilitates recursive constructions, and ensures operational robustness. In this paper, we study the cyclic connectivity of Cayley graph generated by unicyclic triangle free graphs. Then, we determine the exact value of cyclic connectivity of $UG_{n}$ as $\kappa_{c}(UG_{n})=4n-8$ for ~$n\ge 4 $.

\section{Preliminaries}
The underlying topology of an interconnection network can be modeled as a simple, undirected, connected graph $G=(V(G),E(G))$, where $V(G)$ and $E(G)$ denote the vertex set and edge set of $G$, respectively. For a vertex $v\in V(G)$, the \emph{neighbor set} of $v$ in $G$ is defined as $N_{G}(v)=\{u\in V(G)\colon\, uv\in E(G)\}$. For any two vertices $u,v\in V(G)$, we use $cn(u,v)$ to denote the number of common neignbors between $u$ and $v$. The \emph{degree} of $v$, denoted $d_{G}(v)$, is the number of the vertices adjacent to $v$, i.e., $d_{G}(v)=|N_{G}(v)|$. Especially, $\delta(G)=min\left \{ d_{G}(v): v\in V(G) \right \} $, an isolated vertex of $G$ is a vertex $u$ with $d_{G}(u)=0$, a leaf of $G$ is a vertex $u$ with $d_{G}(u)=1$.
For a subset $S\subset V(G)$, we denote $N_{G}(S)$ as the set $\bigcup_{v\in S}N_{G}(v)\setminus S$, and let $N_{G}[S]=N_{G}(S)\cup S$.
For two vertex sets $S$ and $F$, we define $S\setminus F = \{u\in S\colon u\notin F\}$. We denote $C_{n}=u_{1}u_{2}\cdots u_{n}u_{1}$ as an \emph{$n$-cycle}, $P_{n}=u_{1}u_{2}\cdots u_{n}$ as a \emph{path}  with $n-1$ edges, and $K_{1,n}=\{u; v_{1},v_{2},\dots, v_{n}\}$ as a \emph{star}, where $u$ is the center vertex connected to $n$ leaves.
We also denote $g(G)$ as the length of the shortest cycle in $G$. 
For a graph $G$, we say that $H$ is a \emph{subgraph} of $G$, denoted by $H\subseteq G$, provided that $V(H)\subseteq V(G)$ and $E(H)\subseteq E(G)$. For a subset $S\subset V(G)$, we denote $G-S$ as the subgraph obtained by removing all vertices in $S$ from $G$, and $G[S]$ as the subgraph induced by the vertices in $S$.
In a connected graph $G$, a \emph{vertex cut} is a subset $S\subset V(G)$ such that $G-S$ is disconnected. The \emph{connectivity} of $G$ is defined as the minimum cardinality among all possible vertex cuts, represented as $\kappa(G)$. A \emph{$g$-good-neighbor vertex cut} is a subset $S\subset V(G)$ such that $G-S$ is disconnected and each vertex has at least $g$ neighbor in $G-S$. The \emph{$g$-good-neighbor connectivity} of $G$ is defined as the minimum cardinality among all possible $g$-good-neighbor vertex cuts, represented as $\kappa^{g}(G)$. For any $I\subseteq V(G)$, the graph $MC(G-I)$ is the largest component in $G-I$ and $\omega(G-I)$ is  the number of componets in graph $G-I$ if $G-I$ is disconnected.
For convenience, let $[n]=\left\{ 1,2,\dots,n\right\}$ and $[i,j]=\left\{ i,i+1,\dots,j-1,j\right\}$ for $i,j\in [n]$.

Let $\Gamma$ be a group and $S$ be a subset of $\Gamma_{n}\setminus \{1_{\Gamma}\}$, where $1_{\Gamma}$ is the identity of $\Gamma$. \emph{Cayley digraph} $Cay(\Gamma, S)$ is the digraph with vertex set $\Gamma$ and arc set $\{(g, g\cdot s): g\in \Gamma, s\in S\}$. We say that arc $(g, g\cdot s)$ has label $s$. In particular, if $S^{-1}=S$, then $Cay(\Gamma, S)$ is an undirected graph, called \emph{Cayley graph}. Given an symmetric group $Sym(n)$ on $\left\{ 1,2,\dots,n\right\}$ and a set $\mathcal{T}$ of transpositions of $Sym(n)$. Let $G(\mathcal{T})$ be the graph on vertex set $[n]$ and edge set $\left\{ij\colon(ij)\in \mathcal{T}\right\}$. The graph $G(\mathcal{T})$ is called the transposition generating graph of $Cay(Sym(n), \mathcal{T})$. If $G(\mathcal{T})$ is a tree, then $Cay(Sym(n),\mathcal{T})$ is denoted by $\Gamma_{n}$. In particular,  $\Gamma_{n}$ is the star graph $S_{n}$ if $G(\mathcal{T})$ is  isomorphic to a star, $\Gamma_{n}$ is the bubble sort graph $B_{n}$ if $G(\mathcal{T})$ is  isomorphic to a path. If $G(\mathcal{T})$ is a unicyclic triangle free graph, for which the generating graph $G(\mathcal{T})$ has a unique cycle, then denoted the Cayley graph Cay$(Sym(n),\mathcal{T})$ by $UG_{n}$. Specifically, if $\Gamma_{n}$ is a cycle, then $UG_{n}$ is the modified bubble sort graph $MB_{n}$. Additionally, for more graph terminologies and notations not specified here, please refer to~\cite{15}.

\begin{figure}[ht]
\centering
\includegraphics[width=8cm, height=6cm]{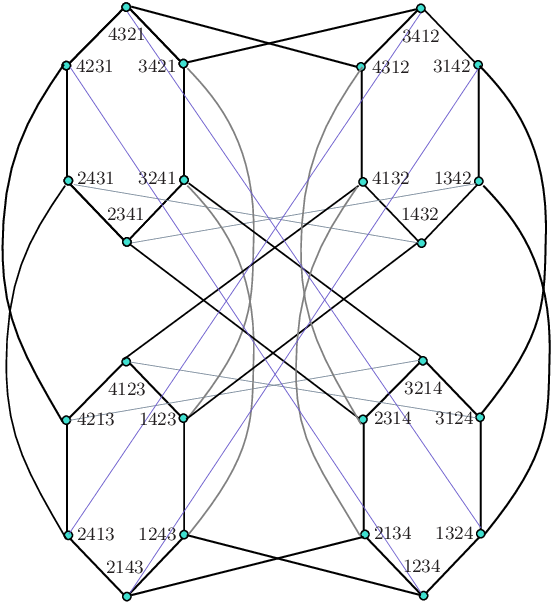}
\caption{The modified bubble-sort graph $MB_{4}$}
\label{fig1}
\end{figure}

Cayley graphs exhibit several ideal properties for interconnection networks, making them highly valuable for research. One of its key features is high fault tolerance, which ensures operational robustness.  The Cayley graph generated by unicyclic triangle free graphs $UG_{n}$, is an undirected graph with $n!$ vertices and $\frac{nn!}{2}$ edges. It is known that $UG_n$ is $n$-regular and bipartite, we proceed to describe its hierarchical structure.

First, we consider that $UG_n$ is not isomorphic to $MB_n$. In this case, $G(\mathcal{T})$ has a unique cycle and at least a leaf. Without loss of generality, assume that $n$ is a leaf of $G(\mathcal{T})$ and $k$ is the unique vertex in $G(\mathcal{T})$ that is adjacent to $n$. Let $\mathcal{T}^{'}=\mathcal{T}\setminus\{(kn)\}$. Obviously, $\mathcal{T}^{'}$ is a generating set of $Sym(n-1)$, and $G(\mathcal{T}^{'})=G(\mathcal{T})-n$ is also a unicyclic triangle free graph with vertex set $[n-1]$. For $i\in[n]$, we denote $UG_{n-1}^{i}$ as the subgraph of $UG_{n}$ induced by all vertices whose rightmost bit is $i$, $UG_{n-1}^{i}$ is an $(n-1)$-dimensional Cayley graph generated by a unicyclic triangle free graph. Consequently, graph $UG_n$ can be decomposed into $n$ vertex-disjoint subgraphs $UG_{n-1}^{i}$ with $i\in[n]$, each of which is isomorphic to $UG_{n-1}$. Let $p=p_{1}p_{2}\cdots p_{n}$ be a permutation on $[n]$ and  $p(kl)$ be a stipulate that the operation interchanging $p_k$ with $p_l$ for $k,l\in [n]$. It can be seen that for any vertex $u\in V(UG_{n-1}^{i})$ with $i\in[n]$, it has exactly one neighbor $u'=u(kn)$ outside of $UG_{n-1}^{i}$, which we refer to as the \emph{out-neighbor} of $u$. An edge that connects a vertex $u$ with its out-neighbor is called
a cross edge. Let $X,Y\subseteq V(G)$, we use $E_{G}(X,Y)$ to denote the set of edges of $G$ with one end in $X$ and the other end in $Y$. Denote $E_{i,j}(UG_{n})=E_{UG_{n}}(V(UG_{n-1}^{i}), V(UG_{n-1}^{j}))$ for $i,j\in[n]$ and $i\neq j$. For any $I\subseteq[n]$, let $UG_{n}^{I}$ be the subgraph of $UG_{n}$ induced by $ {\textstyle \bigcup_{i\in I}^{}}(V(UG_{n}^{i}))$.

Next, we consider that $UG_n$ is isomorphic to $MB_n$. In this case, $G(\mathcal{T})$ is a  $n$-cycle $C_{n}=12\cdots n$. Deleting vertex $n$ from $G(\mathcal{T})$ results in a path $P_{n-1}$, and $\mathcal{T}^{'}=\mathcal{T}\setminus\{(1n),((n-1)n)\}$. Then, $G(\mathcal{T}^{'})$ is the generating set of the $(n-1)$-dimensional bubble sort graph $B_{n-1}$. For $i\in[n]$, we denote $B_{n-1}^{i}$ as the subgraph of $MB_{n}$ induced by all vertices whose rightmost bit is $i$. Consequently, the graph $MB_n$ can be decomposed into $n$ vertex-disjoint subgraphs $B_{n-1}^{i}$ with $i\in[n]$, each of which is isomorphic to $B_{n-1}$. It can be seen that for any vertex $u\in V(B_{n-1}^{i})$ with $i\in[n]$, it has exactly two neighbors $u^{+}=u(1n)$ and $u^{-}=u((n-1)n)$ outside of $B_{n-1}^{i}$, which we refer to as the \emph{out-neighbors} of $u$.

In addition, there are the following useful properties of $UG_{n}$.

\begin{Lem}{\upshape (See~\cite{22}).}  \label{lem2-1}
Let $u$ and $v$ be any two vertices in $UG_{n}$ with $n\ge 4$. Then $cn(u,v)\le 2$.
\end{Lem}

\begin{Lem}{\upshape (See~\cite{22,23}).}  \label{lem2-2}
For $n\geq 4$, $\kappa^{2}(UG_{n})=4n-8$ and $\kappa(B_{n})=n-1$.
\end{Lem}

\begin{Lem}{\upshape (See~\cite{24}).}  \label{lem2-3}
Let $G$ be an $l$-regular $(l\ge3)$ triangle free network with $g(G)=g$. Then $\kappa_c(G)=\kappa^2(G)$ if the following conditions hold:

(1) $\kappa^{2}(G)\le g(l-2)$.

(2) $cn(u,v)\le2$ for any two vertices $u,v\in V(G)$.

(3) $|MC(G-I)|\ge |V(G)|-|I|-g$ for any subset $I\subset V(G)$ with $|I|\le g(l-2)-1$.
\end{Lem}

\begin{Lem}{\upshape (See~\cite{25}).}  \label{lem2-4}
Let $UG_n$ represent a Cayley graph obtained by a transposition unicyclic triangle free graph with order $n\ge 4$. $F$ is a vertex set of $UG_n$. If $|F|\le pn-\frac{p(p+1)}{2}$, where $1\le p\le n-2$, then $UG_n$ has a large connected component, and there are no more than $p-1$ vertices in the remaining components.
\end{Lem}

\begin{Lem}{\upshape (See~\cite{22}).} \label{Lem2-5-1}
Let $C_{4}=u_{1}u_{2}u_{3}u_{4}u_{1}$ be a 4-cycle in $UG_{n}$ with $n\ge 4$. Then $u_{2}=u_{1}(ij),~u_{3}=u_{2}(kl),~u_{4}=u_{3}(ij),~u_{1}=u_{4}(kl)$ for $i,~j,~k,~l\in[n]$ and $i,~j,~k,~l$ differ from each other.
\end{Lem}

\begin{Lem}{\upshape (See~\cite{22,24}).} \label{Lem2-5}
Let $MB_{n}$ be an $n$-dimensional modified bubble-sort graph. The following results hold:

(1) $\kappa(MB_{n})=n$.

(2) $|E_{i,j}(MB_n)|=2(n-2)!$ for any two distinct integers $i, j \in [n]$.

(3) $\{u^+, u^-\}\cap\{v^+, v^-\}=\emptyset$ for any two distinct vertices $u, v\in V(MB_n^i )$ with $i\in[n]$.

(4) $u^+\in V(MB_n^{[3,n]})$ or $u^-\in V(MB_n^{[3,n]})$ for any vertex $u\in V(MB_n^{[2]})$.

\end{Lem}

\begin{Lem}{\upshape (See~\cite{27,28}).} \label{Lem2-6}
Let $MB_{n}$ be an $n$-dimensional modified bubble-sort graph. For $n\ge4$, the following results hold:


(1) Let $p, q, s\in V(MB_n)$ with $pq\in E(MB_n)$, where $p, q, s$ are different from each other. Then $cn(s,p)=0$ or $cn(s,q)=0$.

(2) For any nine distinct vertices $u_i\in V(MB_n)$ with $i\in[9]$, $MB_n$ cannot contain the structure of $A$ (see Fig.\ref{fig2}).
\end{Lem}

%

\begin{figure}[htbp]
\centering
\includegraphics[width=0.3\textwidth]{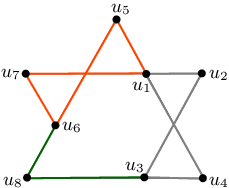}
\caption{The structure of $A$.}
\label{fig2}
\end{figure}

\begin{Lem}{\upshape (See~\cite{27}).}  \label{lem2-7}
Let $MB_{n}$ be an $n$-dimensional modified bubble-sort graph. For any subset $F\subset V(MB_{4})$, the following results hold:

(1) If $|F|\le 5$ and $MB_{4}-F$ is disconnected, then $MB_{4}-F$ has two component, one of which is an isolated vertex.

(2) If $|F|\le 6$ and $MB_{4}-F$ is disconnected, then $MB_{4}-F$ has a large component $C$ with $|V(C)|\ge 4!-|F|-2$.
\end{Lem}

\begin{Lem}\label{Lem2-8}
Let $S$ be a vertex subset of $MB_{n}$.  If $\vert S\vert=4$, then

\begin{center}
$\vert N(S)\vert\ge
\begin{cases}
  & 4n-8~~~~\text{  for } ~n=[4,5], \\
  & 4n-9~~~~\text{  for } ~n\ge6.
\end{cases}$
\end{center}
Moreover, the bound is sharp.
\end{Lem}

\begin{proof}
Let any four vertices $x_{1},x_{2},x_{3}, x_{4}\in V(MB_{n})$ and $S=\left \{ x_{1},x_{2},x_{3},x_{4} \right \} $. Since $MB_{n}$ is bipartite, it contains no 3-cycle, which implies that $\vert E(MB_{n}[S])\vert\le4$. Consider the following cases depending on the cardinality of $\vert E(MB_{n}[S])\vert$.

\textbf{Case 1:} $\vert E(MB_{n}[S])\vert=0$.

It is clear that the vertex set is a set of isolated vertices. According to Lemma \ref{Lem2-6}, any two distinct vertices in $S$ have at most two common neighbors in $MB_{n}$. Then, we proceed with the following case analysis based on the configuration of their common neighbors.

\textbf{Case 1.1:}  There exist at most two pairs $(i, j)$ such that $cn(x_{i},x_{j})=2$.

It is evident that $\vert N_{MB_{n}}(S)\vert=\sum_{i=1}^{4} |N_{MB_{n}}(x_{i})|-|{\textstyle \bigcup_{1\le i<j\le4}^{}}(N_{MB_{n}}(x_{i})\cap N_{MB_{n}}(x_{j}))|\ge 4n-\binom{4}{2}-2 =4n-8$.

\textbf{Case 1.2:} There exist at least three pairs $(i, j)$ such that $cn(x_{i},x_{j})=2$.

Without lose of generality, assume that $cn(x_{1},x_{2})=2$ and $cn(x_{2},x_{3})=2$. By Lemma \ref{Lem2-6}(4), there has no structure $A$ in $MB_{n}$. We immediately have $cn(x_{1},x_{2})=0$. If $cn(x_{i},x_{4})=2$ for some integer $i\in\{1,3\}$, according Lemma \ref{Lem2-6}(4), we have $cn(x_{i},x_{2})=0$. It follows directly that $\sum_{i=1}^{4} |N_{MB_{n}}(x_{i})|-|{\textstyle \bigcup_{1\le i<j\le4}^{}}(N_{MB_{n}}(x_{i})\cap N_{MB_{n}}(x_{j}))|\ge 4n-8$. If $cn(x_{2},x_{4})=2$, according Lemma \ref{Lem2-6}(4), we have $cn(x_{i},x_{4})=0$ for $i\in\{1,3\}$. It follows directly that $\vert N_{MB_{n}}(S)\vert=\sum_{i=1}^{4} |N_{MB_{n}}(x_{i})|-|{\textstyle \bigcup_{1\le i<j\le4}^{}}(N_{MB_{n}}(x_{i})\cap N_{MB_{n}}(x_{j}))|\ge 4n-6$.

\textbf{Case 2:} $\vert E(MB_{n}[S])\vert=1$.

Clearly, the subgraph $MB_{n}[S]$ is isomorphic to $G_{a}$ (see Fig.\ref{fig4}). Since there has no odd cycles in $MB_{n}$, by Lemma \ref{Lem2-6}, we obtain that $cn(x_{1},x_{2})=0$ and $|N_{MB_{n}}(x_{j})\cap N_{MB_{n}}(\{x_{1},x_{2}\})|\le 2$ with $i\in [3,4]$. Moreover,  we have that $cn(x_{3},x_{4})\le 2$. It follows that $\vert N_{MB_{n}}(S)\vert=\sum_{i=1}^{4} |N_{MB_{n}}(x_{i})\setminus S|-|{\textstyle \bigcup_{1\le i<j\le4}^{}}(N_{MB_{n}}(x_{i})\cap N_{MB_{n}}(x_{j}))|\ge 2(n-1)+n+n-6=4n-8$.

\begin{figure}[htbp]
\centering
\includegraphics[width=0.9\textwidth]{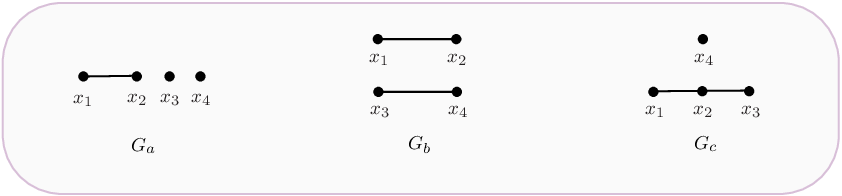}
\caption{ The explanation of  $G_{a}$,$G_{b}$ and $G_{c}$.}
\label{fig4}
\end{figure}

\textbf{Case 3:} $\vert E(MB_{n}[S])\vert=2$.

Clearly, the subgraph $MB_{n}[S]$ is isomorphic to $G_{b}$ or $ G_{c}$ (see Fig.\ref{fig4}).

Scenarios 1: $MB_{n}[S]$ is isomorphic to $G_{b}$. To avoid the occurrence of odd cycles in $MB_{n}$, combining Lemma \ref{Lem2-6}, we find that $cn(x_{i},x_{i+1})=0$ and $|N_{MB_{n}}(x_{j})\cap N_{MB_{n}}(\{x_{3},x_{4}\})|\le 2$, where $i\in\{1,3\}$ and $j\in[2]$. It follows that $\vert N_{MB_{n}}(S)\vert=\sum_{i=1}^{4} |N_{MB_{n}}(x_{i})\setminus S|-|{\textstyle \bigcup_{1\le i<j\le4}^{}}(N_{MB_{n}}(x_{i})\cap N_{MB_{n}}(x_{j}))|\ge 4(n-1)-4=4n-8$.

Scenarios 2: $MB_{n}[S]$ is isomorphic to $G_{c}$(see Fig.\ref{fig4}). Base on the fact that there has no odd cycles and Lemma \ref{Lem2-6}(4), we obtain that there exists at most one integer $i\in[3]$ such that $cn(x_{i},x_{4})=2$ and $cn(x_{j},x_{j+1})=0$ for $j\in [2]$. If $cn(x_{1},x_{4})=2$, then $cn(x_{2},x_{4})=0$. In this case, to avoid the occurrence of the structure of $A$ in $MB_{n}$, we have $|N_{MB_{n}}(x_{3})\cap N_{MB_{n}}(\{x_{1},x_{4}\})|\le1$ except the vertex $x_{2}$. It implies that $|{\textstyle \bigcup_{1\le i<j\le4}^{}}(N_{MB_{n}}(x_{i})\cap N_{MB_{n}}(x_{j}))\setminus \{x_{2}\}|\le3$. If $cn(x_{2},x_{4})=2$, then $cn(x_{i},x_{4})=0$ for $i\in \{1,3\}$. In view of $cn(x_{1},x_{3})\le1$ except the vertex $x_{2}$, we have $|{\textstyle \bigcup_{1\le i<j\le4}^{}}(N_{MB_{n}}(x_{i})\cap N_{MB_{n}}(x_{j}))\setminus \{x_{2}\}|\le3$. It follow that $\vert N_{MB_{n}}(S)\vert=\sum_{i=1}^{4} |N_{MB_{n}}(x_{i})\setminus S|-|{\textstyle \bigcup_{1\le i<j\le4}^{}}(N_{MB_{n}}(x_{i})\cap N_{MB_{n}}(x_{j}))\setminus \{x_{2}\}|\ge n+2(n-1)+n-2-3=4n-7$.

\textbf{Case 4:} $\vert E(MB_{n}[S])\vert=3$.

Clearly, the subgraph $MB_{n}[S]$ is isomorphic to $K_{1,3}$ or $ P_{4}$.

Scenarios 1: $MB_{n}[S]$ is isomorphic to $K_{1,3}$. Without loss of generality, let $K_{1,3}=\{x_{1}; x_{2},x_{3},x_{4}\}$. Since any two vertices of $V(MB_n)$ have at most 2 common neighbors, we obtain $|{\textstyle \bigcup_{1\le i<j\le4}^{}}(N_{MB_{n}}(x_{i})\cap N_{MB_{n}}(x_{j}))\setminus \{x_{1}\}|\le3$. Assume that $|{\textstyle \bigcup_{1\le i<j\le4}^{}}(N_{MB_{n}}(x_{i})\cap N_{MB_{n}}(x_{j}))\setminus \{x_{1}\}|=3$, we obtain that $cn(x_{i},x_{j})=1$ except the vertex $x_{1}$ for $2\le i<j\le4$. By Lemma \ref{Lem2-6}(3), there exist six $i_{j}\in[n]$ with $j\in [6]$ such that $x_{2}=x_{1}(i_{1}i_{2})$,~$x_{3}=x_{1}(i_{3}i_{4})$ ~and ~$x_{4}=x_{1}(i_{5}i_{6})$, which $i_{j}$ differ from each other. If $n\ge6$, then $\vert N_{MB_{n}}(S)\vert= n-3+3(n-3)-3=4n-9$. If $n\in \{4,5\}$, then $|{\textstyle \bigcup_{1\le i<j\le4}^{}}(N_{MB_{n}}(x_{i})\cap N_{MB_{n}}(x_{j}))\setminus \{x_{1}\}|\le 2$, which implies that $\vert N_{MB_{n}}(S)\vert\ge \sum_{i=1}^{4} |N_{MB_{n}}(x_{i})\setminus S|-|{\textstyle \bigcup_{1\le i<j\le4}^{}}(N_{MB_{n}}(x_{i})\cap N_{MB_{n}}(x_{j}))\setminus \{x_{2}\}|\ge (n-3)+3\times (n-1)-2=4n-8$.

Scenarios 2: $MB_{n}[S]$ is isomorphic to $ P_{4}$. Without loss of generality, let $P_{4}=x_{1}x_{2}x_{3}x_{4}$. In this case, to avoid the occurrence of odd cycles in $MB_{n}$, we find that $cn(x_{i},x_{i+1})=cn(x_{1},x_{4})=0$ for $i\in[3]$. Moreover, according to Lemma \ref{Lem2-6}, we have $|N_{MB_{n}-S}(x_{j})\cap N_{MB_{n}-S}(x_{j+2}))|\le 1$ for $j\in[2]$. It follows that $\vert N_{MB_{n}}(S)\vert=\sum_{i=1}^{4} |N_{MB_{n}}(x_{i})\setminus S|-|{\textstyle \bigcup_{1\le i<j\le4}^{}}(N_{MB_{n}}(x_{i})\cap N_{MB_{n}}(x_{j}))|\ge 2(n-1)+2(n-2)-2=4n-8$.

\textbf{Case 5:} $\vert E(MB_{n}[S])\vert=4$.

In view of $g(MB_{n})=4$, it is clear that $MB_{n}[S]$ is a 4-cycle. Therefore, can conclude that $\vert N_{MB_{n}}(S)\vert= 4(n-2)=4n-8$.

Subsequently, we construct the possible structure of $G[S]$ and the distribution of its neighbors for the scenario where equality holds in the inequality, as illustrated in Fig \ref{fig5}. In summary, this lemma holds.
\begin{figure}[htbp]
\centering
\includegraphics[width=0.9\textwidth]{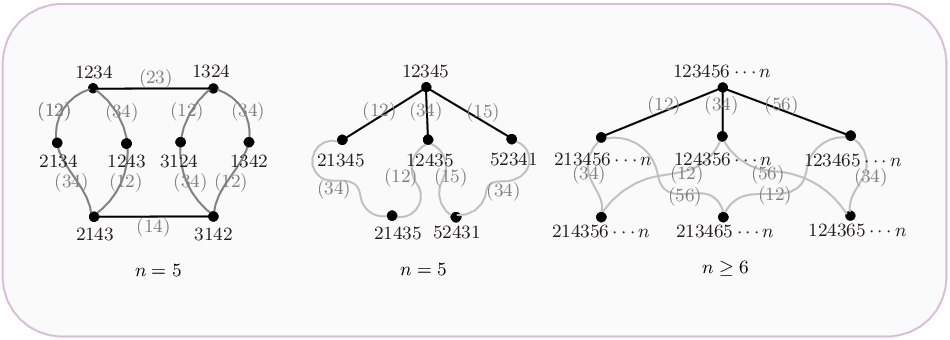}
\caption{ The boundary of inequality is a tight statement}
\label{fig5}
\end{figure}
\end{proof}

\begin{obs}\label{obs2-1}
For each $uv\in E(MB_{4}^{i})$ with $i\in [3]$, we have $|N_{MB_{4}^{4}}(u)|=1$ or $|N_{MB_{4}^{4}}(v)|=1$.
\end{obs}

\begin{Lem}\label{Lem2-9}
Let $F\subset V(MB_{4})$ with $|F|\le 7$. If $MB_{4}-F$ is disconnected, then $MB_{4}-F$ has a large component $C$ with $|V(C)|\ge 4!-|F|-3$.
\end{Lem}

\begin{proof}
Let $F_{i}=F\cap V(MB_{3}^{i})$ for $i \in[4]$. Without lose of generality, note that $|F_{1}|\ge |F_{2}|\ge |F_{3}|\ge |F_{4}|$. It implies $|F_{4}|\le 1$. By Lemma \ref{lem2-1}, we have $\kappa(MB_{3}^{i})=2$ for $i\in [4]$. Base on this fact, it is evident that the subgraph $MB_{3}^{4}-F_{4}$ is connected. For notational convenience, denote by $\mathcal{C}$ the connected component of $MB_{4}-F$ that contains $MB_{3}^{4}-F_{4}$ as a subgraph. According the Lemma \ref{Lem2-5}, we obtain that $|E_{ij}(MB_{4})|=4$ for any two distinct $i,j\in[4]$. If there exists some $k\in[4]$ such that $|F_{k}|\le 1$, we have $|F_{k}|+|F_{4}|\le2 < |E_{ij}(MB_{4})|$, then $MB_{3}^{[k,4]}-F^{[k,4]}$ is connected.

Concluded that $|F_{1}|\ge 2$ and $|F_{3}|\ge 1$. Since $MB_{4}-F$ is disconnected, it is clearly that $|F_{1}|\ge 2$. If $|F_{3}|=0$, then $|F_{4}|=0$ and $MB_{4}^{[3,4]}-F^{[3,4]}$ is connected. For any vertex $u\in V(MB_{3}^{[1,2]})$, Lemma \ref{Lem2-5}(4) states that $u^{+}\in V(MB_{3}^{[3,4]})$ or $u^{-}\in V(MB_{3}^{[3,4]})$. It then follows from this that $MB_{4}-F$ is connected, which is a contradiction.
Based on $|F_{1}|\ge |F_{2}|\ge |F_{3}|\ge |F_{4}|$ and $|F|\le 7$, we can deduce that $2\le|F_{1}|\le 5$, $1\le|F_{2}|\le 3$, $1\le|F_{3}|\le 2$, and $0\le|F_{4}|\le 1$.
We proceed with the following cases analysis based on the cardinality of the vertex set $F_{2}$.

\textbf{Case 1:} $|F_{2}|=1$.

It follows directly that $2\le|F_{1}|\le 5$, $|F_{2}|=|F_{3}|=1$, and $0\le|F_{4}|\le 1$. The validity of $|F_{2}|\le 1$ implies that $MB_{3}^{[2,4]}-F^{[2,4]}$ is connected. Moreover, statement $|V(MB_{3}^{1}-F_{1})|\le 4$ is obvious.  If we suppose $|V(MB_{3}^{1}-F_{1})|= 4$, then by Lemma \ref{Lem2-5}(3), we have $|N_{MB_{4}}(V(MB_{3}^{1}-F_{1}))|\ge 2|V(MB_{3}^{1}-F_{1})|=8>7$, which is a contradiction. This forces us to obtain that $|V(MB_{3}^{1}-F_{1})|\le 3$. Therefore, we conclude that $MB_{4}-F$ has a large component $\mathcal{C}$ with $|V(\mathcal{C})|\ge 4!-|F|-3$.

\textbf{Case 2:} $2\le |F_{2}|\le 3$.

It follows directly that $2\le|F_{1}|\le 4$, $1\le|F_{3}|\le 2$, and $0\le|F_{4}|\le 1$. We now summarize all possible fault vertices distribution scenarios in Table 1 conduct a detailed structural analysis for each case. Obviously, $5\le|F|\le 7$.

\textbf{Case 2.1:} $5\le|F|\leq 6$.
Lemma \ref{lem2-7} guarantees the existence of a large component $\mathcal{C}$ in $G-F$ with $|V(\mathcal{C})| \geq 4!-|F|-2>4!-|F|-3$. This condition is evidently satisfied, thus proving the lemma.

\textbf{Case 2.2:} $|F|=7$.

For clarity, we have categorized this case into five distinct groups, as summarized in Table 1. Each of these groups will be analyzed separately below.

\begin{table}[htbp]
  \centering
  \caption{Cardinality of the Fault Set $F$ and Its Distribution $F_i$ among the $MB_4^i$ Subnetworks}
  \label{tab:fault_distribution_pro}

  \begin{tabular}{>{\columncolor{LightBlue}}c >{\columncolor{LightYellow}}c >{\columncolor{LightOrange}}c c c >{\columncolor{LightGreen}}c c c c c}
    \toprule

    \rowcolor{white}
    $|F|$ & 5 & \multicolumn{3}{c}{6} & \multicolumn{5}{c}{7} \\
    \cmidrule(r){2-2} \cmidrule(lr){3-5} \cmidrule(l){6-10}

    $group$ & \ding{172} & \ding{173} & \ding{174} & \ding{175} & \ding{176} & \ding{177} & \ding{178} & \ding{179} & \ding{180}  \\
     \midrule
    $|F_{1}|$ & 2 & 3 & 2 & 2 & 4 & 3 & 3 & 3 & 2  \\
    \midrule

    $|F_{2}|$ & 2 & 2 & 2 & 2 & 2 & 3 & 2 & 2 & 2  \\
    \midrule

    $|F_{3}|$ & 1 & 1 & 2 & 1 & 1 & 1 & 1 & 2 & 2  \\
    \midrule

    $|F_{4}|$ & 0 & 0 & 0 & 1 & 0 & 0 & 1 & 0 & 1  \\
    \bottomrule
  \end{tabular}
\end{table}

For the cases in groups \ding{176} \ding{177} \ding{178} (see Table 1), it can be seen that $MB_{3}^{[3,4]}-F^{[3,4]}$ is connected and the component $\mathcal{C}$ contains $MB_{3}^{[3,4]}-F^{[3,4]}$ as its subgraph. By Lemma \ref{Lem2-5}(4), there exists a neighbor of any vertex $u\in MB_{3}^{[2]}-F^{[2]}$ in  $MB_{3}^{[3,4]}$. Based on $|F^{[3,4]}|\le 2$, it follows that there are at most two vertices of $ MB_{3}^{i}-F^{i}$ with any $i\in{[2]}$ disconnects to $MB_{3}^{[3,4]}-F^{[3,4]}$. Thus, $|V(MB_{4})-F-V(\mathcal{C})|\le 4$. If $|V(MB_{4})-F-V(\mathcal{C})|= 4$, then $|F|\ge|N_{MB_{4}}(V(MB_{4})-F-V(\mathcal{C}))|\ge 4\times4-8=8>|F|$ by Lemma \ref{Lem2-8}, a contradiction. Thus, $|V(MB_{4})-F-V(\mathcal{C})|\le 3$ and $|V(\mathcal{C})| \geq 4!-|F|-3$.

For the cases in group \ding{179}, we present an exhaustive enumeration of all possible configurations of $MB_{3}^{1}-F_{1}$ and $MB_{3}^{i}-F_{i}$ (where $i \in [2,3]$), illustrated in Fig.\ref{fig3}(a,b,c) and Fig.\ref{fig3}(d,e,f), respectively.

\begin{figure}[htbp]
\centering
\includegraphics[width=0.75\textwidth]{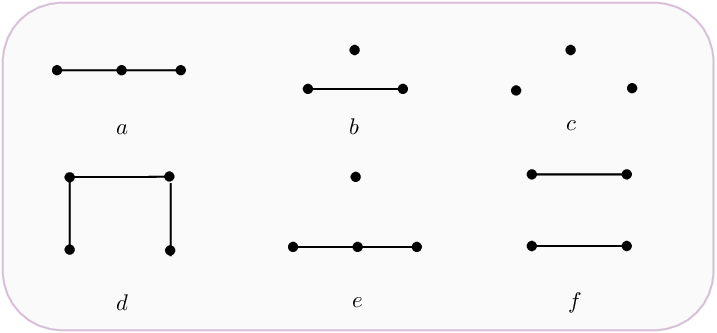}
\caption{The collection ${ MB_{4}^{i} - F_{i} : i \in [3] }$ of all possible structures.}
\label{fig3}
\end{figure}

Since $|F_{4}|=0$, we have $MB_{3}^{4}-F_{4}$ is connected and the component $\mathcal{C}$ contains $MB_{3}^{4}-F_{4}$ as its subgraph. Let $\mathcal{C}_{i}$ be the component containing edges in $MB_{3}^{i}-F_{i}$ with $i\in [3]$. By Lemma \ref{Lem2-5}(1)(3), we obtain that $|E_{14}(MB_{4})|=2(n-2)!=4$ and there exists some vertex $u\in V(MB_{3}^{1}-F_{1})$ connecting to  $MB_{3}^{4}-F_{4}$, which implies $|V(MB_{3}^{1}-F_{1}-V(\mathcal{C})|\le 2$.   For $i\in [2,3]$, according to observation \ref{obs2-1}, it follows that the component $\mathcal{C}_{i}$ is connected to $MB_{3}^{4}-F_{4}$ and $|V(MB_{3}^{i}-F_{i}-V(\mathcal{C}))|\le 1$(see Fig.\ref{fig3} d,e,f). Therefore, $|V(MB_{4}-F-V(\mathcal{C}))|\le 4$. If $|V(MB_{4}-F-V(\mathcal{C}))|= 4$, then $|F|\ge|N_{MB_{4}}(V(MB_{4}-F-V(\mathcal{C})))|\ge 4\times4-8=8>|F|$ by Lemma \ref{Lem2-8}, a contradiction. Thus, $|V(MB_{4})-F-V(\mathcal{C})|\le 3$ and $|V(\mathcal{C})| \geq 4!-|F|-3$.

For the cases in group \ding{180}, we present an exhaustive enumeration of all possible configurations of $MB_{3}^{i}-F_{i}$ with $i \in [3]$(see Fig.\ref{fig3}(d,e,f)).

\textbf{Case 2.2.1:} $MB_{3}^{i}-F_{i}$ is connected for any $i\in [3]$.

Obviously, the subgraph $MB_{3}^{i}-F_{i}$ is isomorphic to $P_{4}$. In view of that $|E_{i4}(MB_{4})|=2(n-2)!=4> |F_{i}|+|F_{4}|$, we obtain that $MB_{4}-F$ is connected, a contradiction.

\textbf{Case 2.2.2:} There exist exactly one $i\in [3]$  such that $MB_{3}^{i}-F_{i}$ is disconnected.

Without lose of generality, let $MB_{3}^{1}-F_{1}$ be disconnected. Obviously, the subgraph $MB_{3}^{i}-F_{i}$ is isomorphic to $P_{4}$ for $i\in [2,3]$. Since $|E_{i4}(MB_{4})|=2(n-2)!=4\ge |F_{i}|+|F_{1}|$ for $i\in [2,3]$, we obtain that $MB_{3}^{[2,4]}-F^{[2,4]}$ is connected and the component $\mathcal{C}$ contains $MB_{3}^{[2,4]}-F^{[2,4]}$ as its subgraph. It follows that $|V(MB_{3}-F-V(\mathcal{C}))|\le |V(MB_{3}^{1}-F_{1})|=4$. If $|V(MB_{4}-F-V(\mathcal{C}))|= 4$, then $|F|\ge|N_{MB_{4}}(V(MB_{4}-F-V(\mathcal{C})))|\ge 4\times4-8=8>|F|$ by Lemma \ref{Lem2-8}, a contradiction. Thus, $|V(MB_{4}-F-V(\mathcal{C}))|\le 3$ and $|V(\mathcal{C})| \geq 4!-|F|-3$.

\textbf{Case 2.2.3:} There exist exactly two $i\in [3]$ such that $MB_{3}^{i}-F_{i}$ is disconnected.

Without lose of generality, let $MB_{3}^{1}-F_{1}$ and $MB_{3}^{2}-F_{2}$ be disconnected. Obviously, the subgraph $MB_{3}^{3}-F_{3}$ is isomorphic to $P_{4}$ and $MB_{3}^{[3,4]}-F^{[3,4]}$ is connected. Moreover, the component $\mathcal{C}$ contains $MB_{3}^{[3,4]}-F^{[3,4]}$ as its subgraph.  Let $\mathcal{C}_{i}$ be a component containing edges in $MB_{3}^{i}-F_{i}$ with $i\in [2]$. Since all possible configurations $d,e,f$ of $MB_{3}^{i}-F_{i}$ have at least two edges, according to observation \ref{obs2-1}, it follows that there is at least one component $\mathcal{C}_{i}$ is connected to $MB_{3}^{4}$ (see Fig.\ref{fig3} d,e,f). Therefore, $|V(MB_{3}^{i}-F_{i}-V(\mathcal{C}))|\le 2$ with $i\in [2]$. Moreover, $|V(MB_{4}-F-V(\mathcal{C}))|\le |V(MB_{3}^{1}-F_{1}-V(\mathcal{C}))|+|V(MB_{3}^{2}-F_{2}-V(\mathcal{C}))|\le 4$. If $|V(MB_{4}-F-V(\mathcal{C}))|=4$, then $|F|\ge|N_{MB_{4}}(V(MB_{4}-F-V(\mathcal{C})))|\ge 4\times4-8=8>|F|$ by Lemma \ref{Lem2-8}, a contradiction. Thus, $|V(MB_{4})-F-V(\mathcal{C})|\le 3$.

\textbf{Case 2.2.4:} $MB_{3}^{i}-F_{i}$ is disconnected any $i\in [3]$.

Obviously, the component $\mathcal{C}$ contains $MB_{3}^{4}-F_{4}$ as its subgraph. Let $\mathcal{C}_{i}$ be a component containing edges in $MB_{3}^{i}-F_{i}$ with $i\in [3]$. Similar to the proof of Case 2.2.3, combining the fact that any $u\in V(MB_{3}^{4})$ has two neighbors in $MB_{3}^{[3]}$, there is at most two component $\mathcal{C}_{i}$ is disconnected to $MB_{3}^{4}$ and $|V(MB_{4})-F-V(\mathcal{C})|\le 4$ (see Fig.\ref{fig3} d,e,f) . Similarly, $|V(MB_{4})-F-V(\mathcal{C})|\neq 4$, Thus, $|V(MB_{4})-F-V(\mathcal{C})|\le 3$.

In summary, the lemma holds.
\end{proof}

\section{Cyclic connectivity of $UG_{n}$}

In this section, we investigate the cyclic connectivity of $UG_n$.

\begin{Lem}\label{lem3-1}
For $n\ge 4$, $\kappa_{c}(UG_n)\le 4n-8$.
\end{Lem}

\begin{proof}
Suppose that $F$ be a minimum 2-good neighbor cut with $|F|=4n-8$. Obviously, the graph $UG_n-F$ is disconnected and the number of its components  $\omega(UG_n-F)\ge 2$. Moreover, any vertex $u\in UG_n-F$ has at two neighbors in $UG_n-F$. That is, $\delta(UG_n-F)\ge 2$. According to the relationship between minimum degree and cycles, we know that every component of $UG_n-F$ contains at least a cycle. In view of $\omega(UG_n-F)\ge 2$, there are at two components containing cycles in $UG_n-F$. Therefore, the vertex subset $F$ is a cyclic vertex cut of $UG_n$ and $\kappa_{c}(UG_n)\le|F|=4n-8$.
\end{proof}

\begin{thm}\label{lem3-2}
For $n\ge 4$, $\kappa_{c}(UG_{n})=4n-8$.
\end{thm}

\begin{proof}
Since the 2-good neighbor connectivity $\kappa^2(UG_n)=4n-8$ and the girth $g(UG_n)=4$, based on the fact that $UG_n$ is $n$-regular, we obtained $\kappa^2(UG_n)=4n-8=g(n-2)$. For $n\ge 6$, given a subset $F\subset V(UG_{n})$ with $|F|\le g(n-2)-1=4n-9$, we have $|F|\le 5n-15=5n-\frac{5\times6}{2}$. By Lemma~\ref{lem2-4}, $UG_{n}-F$ has one large component $MC(UG_{n}-F)$ and a number of small components with at most $4$ vertices in total. Obviously, we get $V(MC(UG_{n}-F))\ge |V(UG_{n})|-|F|-4$. Combining Lemma \ref{lem2-1} with Lemma \ref{lem2-3}, the graph $UG_{n}$ satisfies all conditions of Lemma \ref{lem2-3}. According to Lemma \ref{lem2-2} and Lemma \ref{lem2-3}, we have $\kappa_{c}(UG_{n})=\kappa^{2}(UG_{n})=4n-8$.

In the case $n=5$, based on the fact that $UG_{n}$ has no odd cycles, we find that $G(\mathcal{T})$ connects and it contains a 4-cycle and a leaf. By the hierarchical structure of $UG_n$, the graph $UG_5$ can be decomposed into 5 vertex-disjoint subgraphs $UG_{4}^{i}$ with $i\in[5]$ and $UG_{4}^{i}$ is isomorphic to $MB_4$, where $i\in [5]$.

Suppose to the contrary that there is a cyclic vertex cut $F\subset V(UG_{n})$ such that $|F|\le 4n-9=11$. Let $F_{i}=F\cap UG_{4}^{i}$ for $i\in[5]$. For convenience, note that $I=\{i\in [5]\colon UG_{4}^{i}-F_i ~is~ disconnected\}$, $J=[5]\setminus I$, $F_{I}=\bigcup_{i\in I}F_{i}$, $F_{J}=\bigcup_{j\in J}F_{j}$, $UG_{4}^{I}=\bigcup_{i\in I}UG_{4}^{i}$, $UG_{4}^{J}=\bigcup_{j\in J}UG_{4}^{j}$. By Lemma \ref{lem2-2}, it is evidence that $\kappa (UG_{4}^{i})=4$ for $i\in [5]$. Assume that $\vert I\vert\ge 3$, it follows that $\vert F\vert\ge 4\times3=12>11\ge \vert F\vert$, a contradiction. Thus, $\vert I\vert\le 2$. By the definition of $J$, $UG_{4}^{j}-F_{j}$ is connected for any $j\in J$. There are $(n-2)!=6$ cross edges between $UG_{4}^{i}$ and $UG_{4}^{j}$  for any two distinct integers $i,j\in J$. Obviously, $UG_{4}^{J}-F_{J}$ is connected. If $|I|=0$, then $|J|=5$. Furthermore, $UG_{4}^{J}-F_J$ is connected, which is a contradiction. Hence, $1\le|I|\le2$. Consider the following cases depending on the cardinality of $I$.

\textbf{Case 1.} $|I|=1$.

Without loss of generality, suppose $I=\{1\}$. It means that $UG_{4}^{J}-F_{J}$ is connected.
Moreover, the vertex subset $F$ is a cyclic vertex cut, it implies that there exists a component containing cycles in $UG_{4}^{1}-F^{1}$ and it does not connect to $UG_{4}^{J}-F_J$, denoted by $\mathcal{C}$. Because each vertex $u\in V(\mathcal{C})$ has a unique neighbor $u_1$ outside of $UG_{4}^{1}$, any two vertices $w,v\in V(UG_{4}^{J}-F_J)$ have no common neighbor in $V(\mathcal{C})$. It is evident that $4\le|F\setminus F_{1}|\le 7$ and $4\le|F_{1}|\le 7$. By Lemma~\ref{Lem2-9}, the graph $UG_{4}^{1}-F_{1}$ has one large component $\mathcal{C}$ with $|V(\mathcal{C})|\ge 4!-|F_{1}|-3$. It implies that  $|V(UG_{4}^{1}-F_{1}-\mathcal{C})|\le 3$. It follows that $|V(\mathcal{C})|\ge|V(UG_{4}^{1})|-|F_{1}|-3\ge 24-7-3=14> |F\setminus F_{1}|$. Clearly, the large component $\mathcal{C}$ connects to $UG_{4}^{J}-F_J$. Given that $F$ is a cyclic vertex cut, we obtain that there exists a 4-cycle in  $UG_{4}^{1}-F_{1}-\mathcal{C}$ and $|V(UG_{4}^{1}-F_{1}-\mathcal{C})|\ge4$,  which is a contradiction.

\textbf{Case 2.} $|I|=2$.

Without loss of generality, set $I=\left \{ 1,2 \right \} $. Since each of $UG_{4}^{1}-F_{1}$ and $UG_{4}^{2}-F_{2}$ is disconnected. Given that $|F|\le 11$, combining Lemma \ref{Lem2-5}, we find that $4\le \vert F_{i}\vert\le 7$ with $i\in [2]$.

\textbf{Case 2.1.} $4\le \vert F_{i}\vert\le6$ for each $i\in[2]$.

For $i\in [2]$, by Lemma \ref{Lem2-6}, the graph $UG_{4}^{i}-F_{i}$ has a component $\mathcal{C}_i$ with $|V(\mathcal{C}_{i})|\ge 4!-|F_{i}|-2$. It implies that  $|V(UG_{4}^{i}-F_{i}-\mathcal{C}_{i})|\le 2$. In view of $|V(\mathcal{C}_{i})|\ge |V(UG_{4}^{i}-F_{i}-2)|\ge 16$, we find that $|V(\mathcal{C}_{i})|-|N(V(\mathcal{C}_{i}))\cap V(UG_{4}^{3-i})|-|F\setminus F_{3-i}|\ge 16-6-7\ge 3$, similar to case 1, the component $\mathcal{C}_{i}$ connects to $UG_{4}^{J}-F_J$. Since the vertex set $F$ is a cyclic vertex subset, there exists a 4-cycle in the graph $UG_{4}^{I}-F_I-\mathcal{C}_{1}-\mathcal{C}_{2}$ and the graph is isomorphic to $C_{4}$. It follows that $|F|\ge |N_{UG_{5}}(V(UG_{4}^{I}-F_I-\mathcal{C}_{1}-\mathcal{C}_{2}))|=4\times(5-2)=12>11=|F|$, a contradiction.

\textbf{Case 2.2.} $|F_{1}|=7$(or $|F_{2}|=7$).

Given $|F|\le 11$ and $|F_{2}|\ge 4$, we have $|F_{2}|=4$. Lemma~\ref{Lem2-6} indicates that $V(UG_{4}^{2}-F_{2})$ has two components, one of which is an isolated vertex $v$. Clearly, the graph $UG_{4}^{2}-F_{2}-\{v\}$ connects to $UG_{4}^{J}-F_J$. Moreover, according to the Lemma~\ref{Lem2-9}, we have $V(UG_{4}^{1}-F_{1})$  has a large component  $\mathcal{C}$ with $|V(\mathcal{C})|\ge 4!-|F|-3$. Clearly, the component $\mathcal{C}$ connects to $UG_{4}^{J}-F_J$.  Let $M=UG_{5}-F-(UG_{4}^{J}-F_J)-(UG_{4}^{2}-F_{2}-\{v\})-\mathcal{C}$. It then follows from the preceding analysis that $|V(M)|\le 4$. Note that $F$ is a cyclic vertex cut, there exists a 4-cycle in $M$. It implies $|V(M)|=4$ and the  4-cycle $M$ does not connect to $UG_{4}^{J}-F_J$. It follows that $|F|\ge |N_{UG_{5}}(V(M))|=4\times(5-2)=12>11=|F|$, a contradiction. So, $\kappa_{c}(UG_{5})\ge 4n-8$. Thus, $\kappa_{c}(UG_{5})=4n-8$.

In the case $n=4$, by the hierarchical structure of $UG_{4}$, we have the graph $UG_{4}$ is isomorphic to $MB_{4}$. We only to proof $\kappa_{c}(MB_{4})\ge 8$. Suppose to the contrary that there is a cyclic vertex cut $F\subset V(MB_{4})$ such that $|F|\le 4n-9=7$. By Lemma \ref{Lem2-9}, the graph $MB_{4}-F$ has a large component $\mathcal{C}$ with $|V(\mathcal{C})|\ge 4!-|F|-3$. It implies that  $|V(MB_{4}-F-\mathcal{C})|\le 3$. However, since the vertex subset $F$ is a cyclic vertex cut, combining this and $g(MB_{4})=4$, we find that there exists a 4-cycle and $|V(MB_{4}-F-\mathcal{C})|\ge 4$, which is a contradiction. So, $\kappa_{c}(MB_{4})\ge 4n-8$.  Thus, $\kappa_{c}(UG_{4})=4n-8$.
\end{proof}

\section{Conclusion}

Disjoint paths are pivotal for improving network transmission efficiency and fault tolerance. The presence of cycles in interconnection networks guarantees redundant paths between vertex pairs, rendering cyclic connectivity an essential metric for evaluating network robustness. This parameter has attracted considerable scholarly attention, achieving fruitful research results. Moreover, cayley graphs exhibit several ideal properties for interconnection networks, making them highly valuable for research. One of its key features is high fault tolerance, which ensures ensures operational robustness. In this paper, we determine the exact cyclic connectivity of $UG_{n}$ as $\kappa_{c}(UG_{n})=4n-8$ for ~$n\ge 4$.

Furthermore, we compare the cyclic connectivity of $UG_{n}$ with other connectivity parameters, it shows that the cyclic connectivity of $UG_{n}$ has higher reliability. Our future work will explore the cyclic connectivity of interconnection networks with more complex structures and broader application prospects. It would be beneficial to provide valuable insights into measuring the fault tolerance of interconnection networks.

\section*{Acknowledgements}

This work was supported by the National Science Foundation of China (Nos.12261074 and 12461065), the Middle-aged and Young Research Fund of Qinghai Normal University (No. 2025QZR11).

\section*{Declaration of competing interest}

The authors declare that they have no known competing financial interests or personal relationships that could have
appeared to influence the work reported in this paper.

\section*{Data availability}

No data was used for the research described in the article.


\end{document}